\documentclass[11pt]{amsart}
\usepackage{amsmath,amsthm,amssymb,enumerate}
\usepackage{hyperref}

\newcommand{\prob}{\stackrel{P}{\longrightarrow}}

\newcommand{\one}{{\bf 1}}

\newcommand{\reals}{{\mathbb R}}
\newcommand{\bbr}{\reals}
\newcommand{\vep}{\varepsilon}
\newcommand{\bbz}{\protect{\mathbb Z}}

\newcommand{\chs}{\cite{chakrabarty:hazra:sarkar:2016} }
\newcommand{\chsns}{\cite{chakrabarty:hazra:sarkar:2016}}

\newtheorem{theorem}{Theorem}[section]
\newtheorem{ansatz}{Ansatz}[section]

\newtheorem{lemma}[ansatz]{Lemma}

\newtheorem{remark}{Remark}[section]

\def\Cov{{\rm Cov}}

\def\Var{{\rm Var}}
\def\E{{\rm E}}

\numberwithin{equation}{section}

\begin{document}
\title[Largest eigenvalue of positive mean Gaussian matrices]
{Largest eigenvalue of positive mean Gaussian matrices}
\dedicatory{
Dedicated to the memory of Prof. K. R. Parthasarathy}

\author{Arijit Chakrabarty}
\address{Theoretical Statistics and Mathematics Unit, Indian Statistical Institute, Kolkata, India}
\email{arijit.isi@gmail.com}
\author{Rajat Subhra Hazra}
\address{Mathematical Institute, Leiden University, Leiden, The Netherlands}
\email{r.s.hazra@math.leidenuniv.nl}
\author{Moumanti Podder}
\address{Department of Mathematics, Indian Institute of Science Education and Research Pune, India}
\email{moumantip3@gmail.com}

\subjclass{Primary 60G15. Secondary 60B20.}
\keywords{Gaussian process,  random matrices, largest eigenvalue}

\begin{abstract}
    This short note studies the fluctuations of the largest eigenvalue of symmetric random matrices with correlated Gaussian entries having positive mean. Under the assumption that the covariance kernel is absolutely summable, it is proved that the largest eigenvalue, after centering, converges in distribution to normal with an explicitly defined mean and variance. This result generalizes known findings for Wigner matrices with independent entries. 

    Gaussian distributions in finite and infinite dimensions fascinated Prof. K. R. Parthasarathy or KRP as he was lovingly called. From the vast ocean of subjects of interest to KRP, the authors chose this problem on Gaussian random matrices to pay their tributes to him. KRP will remain an inspiration forever.
\end{abstract}
\maketitle
\section{Introduction}
Wigner matrices with independent Gaussian entries (real or complex) have long been a foundation of random matrix theory. The joint distribution of the eigenvalues of such random matrices can be explicitly written down, allowing many interesting properties like the empirical distribution, the largest eigenvalue, gaps between eigenvalues, and local laws to be derived from the closed-form expression. An extension that has gathered considerable attention is the generalization of these results under minimal assumptions on the distribution of the entries, while still maintaining their independence. Such results fall under the category of universality theorems in random matrix theory. For further details on these classical results, the reader is referred to \cite{AGZ:book}, and to \cite{Erdos-Yau:book} for recent developments in local limit laws and their universality. There has also been various works on relaxing the independence assumptions, especially when considering the empirical spectral distribution of the matrices (\cite{AZ:ESD, Banna:Florence:Magda, chakrabarty:hazra:sarkar:2016, Ajanki:Erdos:Kruger, Che}). Additionally, there have been recent developments concerning bounds on the spectral norm of Wigner matrices with dependent entries (\cite{Erdos:Kruger:Schroder, Jana}). These works consider general entries with summability conditions on the correlations. It is well known that in the independent case with zero-mean entries, and under some moment conditions, the largest eigenvalue of the matrix scaled by $\sqrt{N}$ converges to $2$ almost surely \cite{bai1988necessary} and, after appropriate scaling and centering, converges weakly to the Tracy-Widom law \cite{Lee:Yin}. 

The classical result relevant to this article is the following. When the entries are independent, uniformly bounded, have positive mean $\mu$ and variance of the off-diagonal entries is $\eta^2$, it was proved in \cite{furedi1981eigenvalues} that the largest eigenvalue $\lambda_1$ of the matrix satisfies 
\begin{equation}\label{eq.FK}
\lambda_1 - N\mu \Rightarrow N\left(\frac{\eta^2}{\mu}, 2\eta^2\right), \quad N \to \infty.    
\end{equation}
The aim of this short note is to provide a proof of this result for a Gaussian random matrix with correlated entries under the assumption that correlations are absolutely summable and the sum is non-zero. The result of F{\"u}redi and Koml{\'o}s (\cite{furedi1981eigenvalues}) has been extended in various directions, one of which is applications in the study of the largest eigenvalue of the adjacency matrix of the Erd\H{o}s-R\'enyi random graph (\cite{erdHos2013spectral}) and in its inhomogeneous extensions (\cite{chakrabarty2020eigenvalues}). The approach followed in this article resembles that of these random graphs, though addressing the lack of independence among the entries is new.

\medskip\noindent
\paragraph{\bf Outline:} In Section \ref{sec:setup}, the model is defined and the main result (Theorem \ref{t1}) on the maximum eigenvalue is stated. In Section \ref{sec:proof}, the proof of the main result is provided.

\section{The set-up}\label{sec:setup}
The process considered here is a mean-shifted version of that in \chsns. Let $(Z_{ij}:i,j\in\bbz)$ be a stationary Gaussian process with mean $\theta$, where $\theta>0$ is fixed. Let $A_N$ be an $N\times N$ symmetric random matrix defined by
\begin{equation}\label{def:Amatrix}
A_N(i,j)=Z_{ij}+Z_{ji}\,,1\le i,j\le N\,.
\end{equation}
The goal is to study the asymptotic behaviour of the largest eigenvalue of $A_N$,  which will be denoted by $\lambda_1(A_N)$,  after a centering and suitable assumptions on $(Z_{ij})$.  Interestingly,  $\lambda_1(A_N)$ will require no scaling in the regime considered below.

Let $R(\cdot,\cdot)$ be the covariance kernel of $(Z_{ij})$,  defined by
\[
R(i,j)=\Cov(Z_{00},Z_{ij})\,,i,j\in\bbz\,.
\]
The following is the main assumption.

\noindent\textbf{Assumption.} The covariance kernel $R$ is absolutely summable, that is,
\begin{equation}
\label{setup.eq1}\sum_{i, j=-\infty}^\infty|R(i,j)|<\infty\,,
\end{equation}
and the sum is non-zero, that is,
\begin{equation}
\label{setup.eq2} \sum_{i,j=-\infty}^\infty R(i,j)\neq0\,.
\end{equation}

The following is the main result.

\begin{theorem}\label{t1}
Under the stated assumption,
\[
\lambda_1(A_N)-2N\theta\Rightarrow N\left(\alpha,\sigma^2\right)\,, N\to\infty\,,
\]
where
\[
\alpha=\frac1{2\theta}\sum_{i=-\infty}^\infty \big( R(i,0)+R(0,i)\big)\,,
\]
and
\[
\sigma^2=4\sum_{i,j=-\infty}^\infty R(i,j)\,.
\]
\end{theorem}

\begin{remark} If \eqref{setup.eq2} weren't true, then statement of the above result would change to
\[
\lambda_1(A_N)-2N\theta\prob0\,,N\to\infty\,.
\]
\end{remark}
\begin{remark}
    In the special case where $R(0,0)=\eta^2/2$ and $R(i,j)=0$ if either $i\neq0$ or $j\neq0$, $A_N$ is a classical Wigner matrix whose off-diagonal entries are from $N(0,\eta^2)$ and diagonal entries follow $N(0,2\eta^2)$. In this case, $\sigma^2=2\eta^2$ and $\alpha=(2\theta)^{-1}\eta^2$. Putting $\mu=2\theta$ in \eqref{eq.FK}, which is indeed the common mean of the entries of $A_N$ in this case, Theorem \ref{t1} can easily be seen to be consistent with the result of \cite{furedi1981eigenvalues}. 
\end{remark}

\section{The proof}\label{sec:proof} The following lemmas will have to be proved first.  Let 
\begin{equation}
\label{proof.eq1}W_N=A_N-\E(A_N)\,,
\end{equation}
that is,  the $(i,j)$-th entry of $W_N$ is $X_{ij}+X_{ji}$ for all $1\le i,j\le N$,  where
\[
X_{ij}=Z_{ij}-\theta\,,i,j\in\bbz\,.
\]
Theorem 2.4 of \chs shows that the limiting spectral distribution (LSD) of $N^{-1/2}W_N$,  as $N\to\infty$,  is a compactly supported probability measure.  The first lemma is based on the intuition that the largest eigenvalue converges to the right endpoint of the support of the LSD, and likewise the smallest one goes to the left endpoint. In what follows, for an $N\times N$ matrix $M$, $\|M\|$ is its operator norm defined by
\[
\|M\|=\max_{x\in S^{N-1}}\|Mx\|\,,
\]
where $S^{N-1}=\{x\in\bbr^N:\|x\|=1\}$ and on $\bbr^N$, $\|\cdot\|$ is the usual Euclidean norm.

\begin{lemma}\label{l1}
The family $\{N^{-1/2}\|W_N\|:N=1,2,\ldots\}$ is stochastically tight, that is,
\[
\|W_N\|=O_p\left(\sqrt N\right)\,.
\]
\end{lemma}

\begin{proof} Here and henceforth,  $N$ in the subscript of $W_N$ is suppressed.  
The spectral theorem for Hermitian matrices implies
\begin{equation}
\label{l1.eq1}\|W\|=\max_{v\in S^{N-1}}|v'Wv|\,.
\end{equation}
For a fixed $0<\vep<1/2$,  Lemma 1.4.2 of \cite{chafai:guedon:lecue:pajor:2012} shows that  there exists $S(\vep,N)\subset S^{N-1}$ with
\begin{equation}
\label{l1.eq2}\#S(\vep,N)\le\left(1+\frac2\vep\right)^N,
\end{equation}
where $\#$ denotes cardinality, and
\[
S^{N-1}\subset\bigcup_{x\in S(\vep,N)}\{y\in\bbr^N:\|y-x\|\le\vep\}\,.
\]

Standard arguments show that
\[
\|W\|\le(1-2\vep)^{-1}\max_{x\in S(\vep, N)}|x'Wx|;
\]
see \cite[Exercise 4.4.3(b)]{vershynin2018high}, for example, for the details.

The above inequality combined with \eqref{l1.eq2} and a union bound implies that for $t>0$,
\begin{align}\notag
P\left(\|W\|>t\sqrt N\right)&\le\left(1+\frac2\vep\right)^N\max_{x\in S^{N-1}}P\left(|x'Wx|>t(1-2\vep)\sqrt N\right)\\
\label{l1.eq3}&\le\left(1+\frac2\vep\right)^Ne^{-\delta t^2(1-2\vep)^2N}\max_{x\in S^{N-1}}\E\left(e^{\delta(x'Wx)^2}\right)\,,
\end{align}
the second line following from the Markov inequality,  where $\delta>0$ will be chosen later.  

For all $x=(x_1,\ldots,x_N)\in S^{N-1}$,  $x'Wx$ is a zero mean normal random variable whose variance can be estimated as follows.  Letting $x_i=0$ for all $i\in\bbz\setminus\{1,\ldots,N\}$,  observe
\begin{equation}
\label{l1.eq4} x'Wx=2\sum_{i\in\bbz}\sum_{j\in\bbz}x_ix_jX_{ij}\,,
\end{equation}
and hence
\begin{align*}
\Var\left(x'Wx\right)
&=4\Var\left(\sum_{i, j \in\bbz}x_ix_jX_{ij}\right)\\
&=4\sum_{i, j,k, l\in\bbz}x_ix_jx_kx_lR(i-k,j-l)\\
&\le4\sum_{i, j,k,l\in\bbz}\left|x_ix_jx_kx_lR(i-k,j-l)\right|\\
&=4\sum_{u, v\in\bbz}|R(u,v)|\sum_{i\in\bbz}\left|x_ix_{i+u}\right|\sum_{j\in\bbz}\left|x_jx_{j+v}\right|\\
&\le4\sum_{u, v\in\bbz}|R(u,v)|\,,
\end{align*}
the inequality in the last line being implied by Cauchy-Schwartz and the fact $\sum_ix_i^2=1$.  

Choose 
\[
0<\delta\le\frac1{16}\left(\sum_{u, v\in\bbz}|R(u,v)|\right)^{-1}\,,
\] 
which is possible as the extreme right hand side above is positive because of \eqref{setup.eq1}.  The choice of $\delta$ ensures
\[
\delta\,\Var(x'Wx)\le\frac14\,,
\]
and hence,
\begin{align*}
\E\left(e^{\delta(x'Wx)^2}\right)&=\left(1-2\delta\, \Var(x'Wx)\right)^{-1/2}
\le\sqrt2\,.
\end{align*}
Fixing $\delta$ as above and plugging this in \eqref{l1.eq3}, we get
\[
\lim_{N\to\infty}P\left(\|W\|>t\sqrt N\right)=0\,,\text{ if }t>(1-2\vep)^{-1}\delta^{-1/2}\log^{1/2}(1+2/\vep)\,,
\]
which completes the proof.  
\end{proof}

For the next two lemmas,  $\one_N$ denotes the column vector of length $N$ with each entry $1$.  As usual,  `$N$' will be suppressed to keep the notation simple.

\begin{lemma}\label{l2}
As $N\to\infty$,
\[
N^{-2}\Var\left(\one'W\one\right)\to \sigma^2\,.
\]
\end{lemma}

\begin{proof}
It follows from \eqref{l1.eq4} that
\begin{align}\notag
\Var\left(\one'W\one\right)&=4\Var\left(\sum_{i, j=1}^NX_{ij}\right)\\
&\notag=4\sum_{i, j,k,l=1}^N\Cov(X_{ij},X_{k,l})\\
\label{l1.eq5}&=4\sum_{i, j,k,l=1}^NR(i-k,j-l)\\
\notag&=4\sum_{u=1-N}^{N-1}\sum_{v=1-N}^{N-1}R(u,v)\\
\notag&\hspace{1in}\#\{(i,j,k,l)\in\{1,\ldots,N\}^4:i-k=u,j-l=v\}\\
\notag&=4\sum_{u, v=1-N}^{N-1}(N-|u|)(N-|v|)R(u,v)\,.
\end{align}
Thus,
\begin{equation}
\label{l1.eq6}\left|N^{-2}\Var\left(\one'W\one\right)-\sigma^2\right|\le4\sum_{u, v\in\bbz}\left(\frac{|u|}N\wedge1\right)\left(\frac{|v|}N\wedge1\right)|R(u,v)|\,.
\end{equation}
As $N\to\infty$,  \eqref{setup.eq1} implies the right hand side goes to zero.  Hence the proof follows.
\end{proof}

The above lemma and \eqref{setup.eq2} ensure $\sigma^2>0$.  The proof of the next lemma also follows from a similar mean and variance calculation.

\begin{lemma}\label{l3}
As $N\to\infty$,
\[
N^{-2}\one'W^2\one\prob2\alpha\theta\,.
\]
\end{lemma}

\begin{proof}  The proof would follow once it can be shown that
\begin{equation}
\label{l3.eq1} \lim_{N\to\infty}N^{-2}\E\left(\one'W^2\one\right)=\sum_{i=-\infty}^\infty \big( R(i,0)+R(0,i)\big)\,,
\end{equation}
and
\begin{equation}
\label{l3.eq2}\Var\left(\one'W^2\one\right)=o(N^4)\,.
\end{equation}
For $N$ fixed,
\begin{align}\notag
\one'W^2\one&=\sum_{i, j,k=1}^N W(i,j)W(j,k)\\
\label{l3.eq3}&=\sum_{i, j,k=1}^N\left(X_{ij}X_{kj}+X_{ji}X_{jk}\right)+2\sum_{i, j,k=1}^N X_{ij}X_{jk}\,.
\end{align}

Proceeding towards \eqref{l3.eq1}, write
\begin{align*}
\sum_{i, j,k=1}^N\E\left(X_{ij}X_{kj}\right)&=\sum_{i, j,k=1}^N\Cov(Z_{ij},Z_{kj})\\
&=N\sum_{i, k=1}^NR(i-k,0)\,.
\end{align*}
An argument similar to \eqref{l1.eq5}-\eqref{l1.eq6} shows
\[
\lim_{N\to\infty}N^{-2}\sum_{i, j,k=1}^N\E\left(X_{ij}X_{kj}\right)=\sum_{i=-\infty}^\infty R(i,0)\,,
\]
and in the same vein it can be shown that
\[
\lim_{N\to\infty}N^{-2}\sum_{i,j,k=1}^N\E\left(X_{ji}X_{jk}\right)=\sum_{i=-\infty}^\infty R(0,i)\,.
\]
Combining these with \eqref{l3.eq3} and the observation that 
\begin{align*}
\left|\sum_{i, j,k=1}^N\E(X_{ij}X_{jk})\right|&=\left|\sum_{i, j,k=1}^N R(i-j,j-k)\right|\\
&\le\sum_{j=1}^N\sum_{i,k=-\infty}^\infty|R(i-j,j-k)|\\
&=N\sum_{u, v=-\infty}^\infty|R(u,v)|\,,
\end{align*}
and hence \eqref{l3.eq1} follows.

For showing \eqref{l3.eq2}, expand the expectation of the square of the first term in the first summation of \eqref{l3.eq3} with the help of Wick's formula (\cite[Chapter 1, Theorem 1]{mingo2017free}) as
\begin{align*}
&\E\left[\left(\sum_{i, j,k=1}^N X_{ij}X_{kj}\right)^2\right]\\
&=\sum_{i, j,k=1}^N\sum_{i', j', k'=1}^N\E\left(X_{ij}X_{kj}X_{i'j'}X_{k'j'}\right)\\
&=\sum\Big(R(i-k,0)R(i'-k',0)+R(i-i',j-j')R(k-k',j-j')\\
&\hspace{1in}+R(i-k',j-j')R(k-i',j-j')\Big)\,,
\end{align*}
where the sum in the penultimate line is over all $i,j,k,i',j',k'$ from $1$ to $N$. An argument, which has been  used already a few times,  shows that
\[
\lim_{N\to\infty}N^{-4}\E\left[\left(\sum_{i, j,k=1}^N X_{ij}X_{kj}\right)^2\right]=\left(\sum_{i=-\infty}^\infty R(i,0)\right)^2\,.
\]
The details of the proof in a more general situation is given in the arXiv version of \cite{chakrabarty:hazra:sarkar:2015}; see (4.33) in the proof of Theorem 2.1 therein. A similar calculation can be done for the other terms in the expansion of the square of \eqref{l3.eq3}, which would show that
\[
\lim_{N\to\infty}N^{-4}\E\left[(\one'W^2\one)^2\right]=\left[\sum_{i=-\infty}^\infty \big( R(i,0)+R(0,i)\big)\right]^2\,.
\]
This in conjunction with \eqref{l3.eq1} establishes \eqref{l3.eq2} and thus completes the proof.
\end{proof}

\begin{proof}[Proof of Theorem \ref{t1}]
Rewrite \eqref{proof.eq1} as
\[
A_N-2\theta\one\one'=W,
\]
where $W=W_N$. Since $\theta>0$, the largest eigenvalue of $2\theta\one\one'$ is $2N\theta$, which implies 
\[
|\lambda_1-2N\theta|\le\|W\|\,,
\]
where $\lambda_1=\lambda_1(A_N)$.  Dividing both sides by $N$ and using Lemma \ref{l1}, we get
\begin{equation}
\label{t1.eq1} N^{-1}\lambda_1\prob2\theta\,,N\to\infty\,.
\end{equation}

Let $v$ be an eigenvector (normalized to have norm $1$) of $A_N$ corresponding to its largest eigenvalue,  satisfying
\begin{equation}
\label{proof.eq0}\one'v\ge0\,.
\end{equation}
Thus,
\begin{equation}
\label{proof.eq2}\lambda_1v=A_Nv=Wv+2\theta(\one'v)\one\,.
\end{equation}
Pre-multiply both sides by $v'$ and use the fact that $v$ is a unit-norm vector to get
\[
\lambda_1=v'Wv+2\theta(\one'v)^2\,.
\]
Dividing throughout by $N$ and using \eqref{t1.eq1} in conjunction with Lemma \ref{l1} yield
\[
N^{-1}2\theta(\one'v)^2\prob2\theta.
\]
Divide both sides by $2\theta$ and take the positive square root to conclude
\begin{equation}
\label{proof.eq3}N^{-1/2}(\one'v)\prob1,
\end{equation}
\eqref{proof.eq0} ensuring that $\one'v$ appears in the left hand side when the positive root is taken.

A consequence of \eqref{proof.eq2} is
\[
(\lambda_1I-W)v=2\theta(\one'v)\one\,.
\]
Lemma \ref{l1} and \eqref{t1.eq1} together show 
\begin{equation}
\label{t1.eq2}\lambda_1^{-1}\|W\|\prob0\,.
\end{equation}
Hence $(\lambda_1I-W)$ is invertible whp (with high probability, that is, probability goes to $1$ as $N\to\infty$).  

Therefore,  whp,
\[
v=2\theta(\one'v)(\lambda_1I-W)^{-1}\one\,.
\]
Pre-multiply both sides by $\one'$ to get
\[
\one'v=2\theta(\one'v)\one'(\lambda_1I-W)^{-1}\one\,.
\]
It follows from \eqref{proof.eq3} that $\one'v\neq0$ whp.  Thus,
\[
1=2\theta\one'(\lambda_1I-W)^{-1}\one\,,
\]
which implies 
\[
\lambda_1=2\theta\one'(I-\lambda_1^{-1}W)^{-1}\one\,.
\]
Use \eqref{t1.eq2} again to expand the right hand side to get
\begin{align}\notag
\lambda_1&=2\theta\one'\left(\sum_{n=0}^\infty\lambda_1^{-n}W^n\right)\one\\
\notag&=2\theta\one'\left(\sum_{n=0}^2\lambda_1^{-n}W^n+\lambda_1^{-3}W^3(I-\lambda_1^{-1}W)^{-1}\right)\one\\
\label{eq.expansion}&=2N\theta+2\theta\lambda_1^{-1}\one'W\one+2\theta\lambda_1^{-2}\one'W^2\one\\
\notag& \hspace{1in}+2\theta\lambda_1^{-3}\one'W^3(I-\lambda_1^{-1}W)^{-1}\one\,.
\end{align}

Combine \eqref{t1.eq1} with Lemma \ref{l2} to get
\begin{equation}
\label{t1.eq3} 2\theta\lambda_1^{-1}\one'W\one\Rightarrow N(0,\sigma^2)\,,N\to\infty\,.
\end{equation}
Similarly,  \eqref{t1.eq1} and Lemma \ref{l3} together imply 
\begin{equation}
\label{t1.eq4} 2\theta\lambda_1^{-2}\one'W^2\one\prob\alpha\,, N\to\infty\,.
\end{equation}
Finally, since $\|\cdot\|$ is the operator norm, which is sub-multiplicative, and $\one$ is a vector of norm $\sqrt N$, we get
\begin{align*}
\left|\lambda_1^{-3}\one'W^3(I-\lambda_1^{-1}W)^{-1}\one\right|&\le|\lambda_1|^{-3}N\|W\|^3\|(I-\lambda_1^{-1}W)^{-1}\|\\
&=O_p\left(N^{-2}\|W\|^3\right)\\
&\prob0\,,N\to\infty\,,
\end{align*}
\eqref{t1.eq1} and \eqref{t1.eq2} implying the second line and the last line following from Lemma \ref{l1}. Plug \eqref{t1.eq3} and \eqref{t1.eq4} in \eqref{eq.expansion} to get
\[
\lambda_1-2N\theta\Rightarrow N(\alpha,\sigma^2)\,,
\]
as desired.
\end{proof}

\begin{remark}
A close look at the above proof reveals that it goes through even if the claim of Lemma \ref{l1} is reduced to 
\[
\|W\|=O_p\left(N^{1-\vep}\right)
\]
for some $\vep>0$. In other words, \eqref{setup.eq1} can possibly be relaxed to something under which a claim similar to that of Theorem \ref{t1}, albeit with a scaling and a different centering, still holds. In this situation, \eqref{eq.expansion} will have to be replaced by
\[
\lambda_1=2\theta\sum_{i=0}^n\lambda_1^{-i}\one'W^i\one+2\theta\lambda_1^{-n-1}\one'W^{n+1}(I-\lambda_1^{-1}W)^{-1}\one\,,
\]
where $n=[1/\vep]$, that is, the largest integer less than or equal to $1/\vep$. 

Such scenarios, and the behavior of the largest eigenvalue under more general correlation structures, are explored in a recent preprint \cite{banerjee2024edge}.
\end{remark}

\section*{Acknowledgement} The authors thank an anonymous referee for a careful reading of the manuscript and suggesting changes that have helped to improved it. The research of the second author was supported through NWO Gravitation Grant NETWORKS 024.002.003. The research of the third author is supported by SERB/CRG/2021/006785, from the Science and Engineering Research Board, Government of India.


\end{document}